\documentclass[letterpaper,11pt]{article}

\usepackage{amsmath,amsfonts,amssymb,amsthm,mathrsfs,dsfont}
\usepackage{cases}
\usepackage{enumerate}
\usepackage[usenames,dvipsnames]{color}
\usepackage{verbatim} 
\usepackage{array} 
\usepackage{graphicx,epstopdf} 
\usepackage{mathtools}

\theoremstyle{plain} \newtheorem{thm}{Theorem}
\theoremstyle{plain} \newtheorem{lemma}[thm]{Lemma}
\theoremstyle{plain} 
\theoremstyle{plain} \newtheorem{cor}[thm]{Corollary}
\theoremstyle{plain} 
\theoremstyle{plain} 
\theoremstyle{definition} 
\theoremstyle{definition} 
\theoremstyle{plain} 

\usepackage[normalem]{ulem}
\newcommand{\patColor}[1]{\ifx \hfuzz#1\hfuzz \textcolor{Cyan}{[]} \else \textcolor{Cyan}{#1} \fi}

\newcommand{\sub}[0]{\subseteq}
\newcommand{\sm}[0]{\setminus}
\renewcommand{\dots}[0]{,\ldots,}
\newcommand{\ra}[0]{\rightarrow}

\newcommand{\beq}[1]{\begin{equation}\label{#1}}
\newcommand{\enq}[0]{\end{equation}}

\newcommand{\bn}[0]{\bigskip\noindent}
\newcommand{\mn}[0]{\medskip\noindent}
\newcommand{\nin}[0]{\noindent}

\newcommand{\0}[0]{\varnothing}
\newcommand{\E}[0]{{\mathbb{E}}}

\newcommand{\C}[2]{{{#1}\choose{{#2}}}}
\newcommand{\Cc}[0]{\tbinom}

\newcommand{\gb}[0]{\beta }
\newcommand{\gc}[0]{\gamma }
\newcommand{\gd}[0]{\delta }

\newcommand{\gl}[0]{\lambda }
\newcommand{\gL}[0]{\Lambda}

\newcommand{\gO}[0]{\Omega}

\newcommand{\gs}[0]{\sigma}

\newcommand{\gz}[0]{\zeta}
\newcommand{\vt}[0]{\vartheta}

\newcommand{\J}[0]{{\cal J}}

\newcommand{\T}[0]{{\cal T}}

\newcommand{\ttt}[0]{t}

\newcommand{\YY}[0]{{\bf Y}}

\newcommand{\pat}[1]{{\color{Blue}{#1}}}

\newcommand{\pr}{\mathbb{P}}
\newcommand{\mean}{\mathbb{E}}
\newcommand{\fit}{\text{fit}}
\newcommand{\eps}{\varepsilon}
\newcommand{\perm}[1]{\mathfrak{S}_{#1}}
\newtheorem*{remark*}{Remark}

\title{Proof of an entropy conjecture of \\ Leighton and Moitra}

\author{
H\"{u}seyin Acan
\footnote{Department of Mathematics, Rutgers University}
\thanks{Supported by National Science Foundation Fellowship (Award No.~1502650).}
\\
{\small \texttt{huseyin.acan@rutgers.edu}}
\and
Pat Devlin \footnotemark[1] \thanks{Supported by NSF grant DMS1501962.}\\
{\small \texttt{prd41@math.rutgers.edu }}
\and
Jeff Kahn \footnotemark[1] \footnotemark[3]   \\
{\small \texttt{jkahn@math.rutgers.edu}}
}
\date{}

\begin{document}
\renewcommand{\thefootnote}{\fnsymbol{footnote}}
\maketitle

\begin{abstract}
We prove the following conjecture of Leighton and Moitra.
Let $T$ be a tournament on $[n]$ and $\perm{n}$ the set of permutations of~$[n]$.
For an arc $uv$ of $T$, let $A_{uv}=\{\sigma \in \perm{n}\,: \, \sigma(u)<\sigma(v) \}$.

\mn
\textbf{Theorem.}
For  a fixed $\eps>0$, if $\pr$ is a probability distribution on $\perm{n}$
such that $\pr(A_{uv})>1/2+\eps$ for every arc $uv$ of $T$, then the binary entropy of $\pr$
is at most $(1-\vt_{\eps})\log_2 n!$ for some (fixed) positive $\vt_\eps$.

\mn
When $T$ is transitive the theorem is due to Leighton and Moitra; for this case
we give a short proof with a better $\vt_\eps$.

\end{abstract}

\section{Introduction}\label{Intro}

In what follows we use $\log $ for $\log_2$ and $H(\cdot)$ for binary entropy.
The purpose of this note is to
prove the following natural statement, which was conjectured by Tom Leighton and Ankur Moitra \cite{L-M}
(and told to the third author by Moitra in 2008).

\begin{thm}\label{LMConj}
Let $T$ be a tournament on $[n]$ and $\sigma$ a random
(not necessarily uniform) permutation of $[n]$ satisfying:
\beq{hyp}
\mbox{for each arc $uv$ of $T$, $~\pr(\gs(u) < \gs(v)) > 1/2+\eps$.}
\enq
Then
\beq{conclusion}
H(\sigma) \leq (1-\vt )  \log n!,
\enq
where $\vt > 0$ depends only on $\varepsilon$.
\end{thm}

\nin
(We will usually think of permutations as
bijections $\gs:[n]\ra [n]$).
The original motivation for Leighton and Moitra came mostly from questions about sorting partially
ordered sets;
see \cite{L-M} for more on this.

\mn

For the special case of \emph{transitive} $T$,
Theorem~\ref{LMConj} was proved in \cite{L-M}
with $\vartheta_\eps =C\eps^4$.
Note that for a \emph{typical} (a.k.a.\ \emph{random}) $T$, the conjecture's
hypothesis is unachievable, since, as shown long ago by Erd\H{o}s and Moon \cite{EM},
no $\gs$ agrees with $T$ on more than a $(1/2+o(1))$-fraction of its arcs.
In fact, it seems natural to expect that transitive tournaments are the \emph{worst} instances,
being the ones for which the hypothesized agreement is easiest to achieve.
From this standpoint, what we do here may be considered somewhat unsatisfactory,
as our $\vt$'s are quite a bit worse than those in \cite{L-M}.
For transitive $T$ it's easy to see \cite[Claim 4.14]{L-M}
that one can't take $\vt$ greater than $2\eps$, which seems likely to be close to the truth.
We make some progress on this, giving a surprisingly simple proof of the following improvement of
\cite{L-M}.

\begin{thm}\label{thm:transitive}
For $T$, $\pr$, $\sigma$ as Theorem~\ref{LMConj} with $T$ transitive,
\[
H(\sigma) \leq (1-\eps^2/8)  n\log n.
\]
\end{thm}

\mn

The proof of Theorem~\ref{LMConj} is given in Section~\ref{Proof} following brief preliminaries
in Section~\ref{Prelim}.
The underlying idea is similar to that of \cite{L-M}, which in turn
was based on the beautiful tournament ranking bound of W. Fernandez de la Vega \cite{delavega};
see Section~\ref{Proof} (end of ``Sketch") for an indication of the relation to~\cite{L-M}.
Theorem~\ref{thm:transitive} is proved in Section~\ref{sec:transitive}.

\section{Preliminaries}\label{Prelim}

\mn
\emph{Usage}

In what follows we assume $n$ is large enough to support our arguments and
pretend all large numbers are integers.

As usual $G[X]$ is the subgraph of $G$ induced by $X$; we use
$G[X,Y]$ for the bipartite subgraph induced
(in the obvious sense) by disjoint $X$ and $Y$.
For a digraph $D$, $D[X]$ and $D[X,Y]$ are used analogously.
For both graphs and digraphs, we use $|\cdot|$ for number of edges (or arcs).

Also as usual, the \emph{density} of a pair $(X,Y)$ of disjoint subsets of $V(G)$
is $d(X,Y) =d_G(X,Y) = |G[X,Y]|/(|X||Y|)$, and
we extend this to bipartite digraphs $D$ in which
\beq{onedirection}
\mbox{\emph{at most one of
$D\cap (X\times Y)$, $D\cap (Y\times X)$ is nonempty.}}
\enq
For a digraph $D$,
$D^r$ is the digraph gotten from $D$ by reversing its arcs.

Write $\perm{n}$ for the set of permutations of $[n]$.  For $\sigma\in \perm{n}$,
we use $T_{\sigma}$ for the corresponding (transitive)
tournament on $[n]$ (that is, $uv\in T_\gs$ iff $\gs(u)<\gs(v)$)
and for a digraph $D$ (on $[n]$) define
\[
\fit(\sigma,D)= |D\cap T_{\sigma}|-|D^r\cap T_{\sigma}|
\]
(e.g.\ when $D$ is a tournament, this is a measure of
the quality of $\gs$ as a ranking of $D$).

\mn
\emph{Regularity}

Here we need just Szemer\'edi's basic notion \cite{Szem}
of a regular pair and a very weak version (Lemma~\ref{LKomlos}) of his Regularity Lemma.
As usual a bipartite graph $H$ on disjoint $X\cup Y$ is
$\gd$-\emph{regular} if
\[
|d_H(X',Y')-d_H(X,Y)|< \gd
\]
whenever $X'\sub X$, $Y'\sub Y$, $|X'|>\gd|X|$ and $|Y'|>\gd|Y|$,
and we extend this in the obvious way
to the situation in \eqref{onedirection}.
It is easy to see that if a bigraph $H$ is $\gd$-regular
then
its bipartite complement is as well; this implies
that for a tournament $T$ on $[n]$ and
$X$, $Y$ disjoint subsets of $[n]$,

\mn
\beq{comp}
\mbox{\emph{$T\cap (X\times Y)$ is $\gd$-regular if and only if $T\cap (Y\times X)$ is}.}
\enq

\medskip
The following statement should perhaps be considered folklore, though
similar results were proved by
J\'anos Koml\'os, circa 1991 (see \cite[Sec.\ 7.3]{Komlos-Simonovits}).
\begin{lemma}\label{LKomlos}
For each $\gd>0$ there is a
$\gb > 2^{-\gd^{-O(1)}}$
such that
for any bigraph $H$ on $X\cup Y$ with $|X|,|Y|\geq n$, there is a
$\gd$-regular pair $(X',Y')$ with $X'\sub X, Y'\sub Y$ and each of $|X'|,|Y'|$
at least $\gb n$.
\end{lemma}

\begin{cor}\label{ternary}
For each $\gd>0$, $\gb$ as in Lemma~\ref{LKomlos}
and digraph $G=(V,E)$,
there is a partition $L\cup R\cup W$ of $V$ such that $E\cap (L\times R)$ is $\delta$-regular and
$
\min\{|L|,|R|\} \ge \beta|V|/2.
$
\end{cor}
\begin{proof}
Let $X\cup Y$ be an (arbitrary) equipartition of $V$ and apply Lemma~\ref{LKomlos}
to the undirected graph $H$ underlying the digraph $G\cap (X\times Y)$.\end{proof}

\section{Proof of Theorem~\ref{LMConj}}\label{Proof}

We now assume that $\gs$ drawn from the probability distribution $\pr$ on $\perm{n}$
satisfies \eqref{hyp} and try to show \eqref{conclusion}
(with $\vt $ TBA).  We use $\E$ for expectation w.r.t. $\pr$
and $\mu$ for uniform distribution on $\perm{n}$.

\mn
\emph{Sketch and connection with} \cite{L-M}

We will produce $S_1\dots S_m\sub T$ with $S_i\sub L_i\times R_i$ for some disjoint $L_i,R_i\sub [n]$,
satisfying:

\begin{itemize}
\item[(i)]  with $\|S_i\|:=\min\{|L_i|,|R_i|\}$,
$\sum\|S_i\| =\gO(n\log n)$ (where the implied constant depends on $\eps$);

\item[(ii)]  each $S_i$ is $\gd$-regular (with $\gd=\gd_\eps$ TBA);

\item[(iii)]  for all $i < j$, either $(L_i\cup R_i)\cap(L_j\cup R_j)=\0$ or $L_j\cup R_j$ is contained in one of $L_i,R_i$
(note this implies the $S_i$'s are disjoint).
\end{itemize}

\mn
Let $A_i=\{\fit(\gs,S_i)>\eps |S_i|\}$ and
$Q=\{\sum\{\|S_i\|:A_i ~\text{occurs}\}   =\gO(n\log n)\}$.
The main points are then:

\begin{itemize}
\item[(a)] $\pr (Q)$ is bounded below by a positive function of $\varepsilon$.
(This is just (i) together with a couple applications of Markov's
Inequality.)
\item[(b)] Regularity of $S_i$ implies $\mu(A_i) \leq \exp[-\gO(\|S_i\|)]$.
\item[(c)] Under (iii), for any $I\sub [m]$,
\[
\mbox{$\mu(\cap_{i\in I}A_i) < \exp[-\sum_{i\in I}\gO(\|S_i\|)]$}
\]
(a weak version of independence of the $A_i$'s under $\mu$).
\end{itemize}

\mn
And these points easily combine to give \eqref{conclusion}
(see \eqref{Hsigma} and \eqref{main}).

\mn

For the transitive case in \cite{L-M} most of this argument is
unnecessary; in particular, regularity disappears and there is a natural \emph{decomposition} of
$T$ into $S_i$'s:  Supposing $T=\{ab:a<b\}$ and (for simplicity) $n=2^k$, we may
take the $S_i$'s to be the sets $L_i\times R_i$ with $(L_i,R_i)$ running over pairs
\beq{TD}
([(2s-2)2^{-j}n+1,(2s-1)2^{-j}n],[(2s-1)2^{-j}n+1,2s2^{-j}n]),
\enq
with $j\in [k]$ and $s\in [2^{j-1}]$.
(As mentioned earlier, this decomposition of the (identity) permutation $(1\dots n)$ also
provides the framework for \cite{delavega}.)
After some translation, our argument (really, a fairly small subset thereof) then specializes to
essentially what's
done in \cite{L-M}.\qed

\bn

Set $\gd = .03 \eps$ and
let $\gb$ be half the $\gb$ of Lemma~\ref{LKomlos} and
Corollary~\ref{ternary}.
We use the corollary to find a rooted tree $\cal T$
each of whose internal nodes has degree (number of children) 2 or 3,
together with
disjoint
subsets $S_1,S_2\dots S_m$ of (the arc set of) $T$, corresponding
to the internal nodes of $\T$.
The nodes of $\cal T$ will be subsets of $[n]$
(so the \emph{size,} $|U|$, of a node $U$ is its size as a set).

To construct $\T$, start with root
$V_1=[n]$ and
repeat the following for $k=1,\ldots$
until each unprocessed node has size less than (say) $\ttt:=\sqrt{n}$.
Let $V_k$ be an
unprocessed node of size at least $\ttt$
and apply Corollary~\ref{ternary} to $T[V_k]$ to produce
a partition $V_k= L_k\cup R_k\cup W_k$,
with $|L_k|, |R_k|>\gb |V_k|$ and
$S_k:=T\cap (L_k\times R_k)$
$\gd$-regular of density at least 1/2.
(Note \eqref{comp} says we can reverse the roles of $L_k$ and $R_k$
if the density of $T\cap (L_k\times R_k)$ is less than 1/2.)
Add $L_k,R_k, W_k$ to $\T$ as the children of $V_k$ and
mark $V_k$  ``processed."
(Note the $V_k$'s are the \emph{internal} nodes of $\T$; nodes of
size less then $\ttt$ are not processed and are automatically leaves.
Note also that there is no restriction on $|W_k|$ and
that, for $k>1$, $V_k$ is equal to one of $L_i$, $R_i$, $W_i$ for some $i<k$.)

Let $m$ be the number of internal nodes of $\T$ (the final tree).
Note that the leaves of $\T$ have size at most $\ttt$ and that
the $S_i$'s satisfy (ii) and (iii) of the proof sketch; that they also satisfy
(i) is shown by the next lemma.

Set
\[
\mbox{$\gL =\sum_{i=1}^m |V_i| $;}
\]
this quantity will play a central role in what follows.

\begin{lemma}\label{sum |V_i|}
$\gL \geq \frac{1}{2}n\log_3n$     
\end{lemma}
\nin
\begin{proof}
This will follow easily from the next general (presumably known) observation, for which we assume
$\T $ is a tree satisfying:

\begin{itemize}

\item
the nodes of $\T $ are subsets of $S$, an $s$-set
which is also the root of $\T $;

\item
the children of each internal node $U$ of $\T $ form a partition of $U$ with at most $b$ blocks;

\item
the leaves of $\T $ are $U_1\dots U_r$, with $|U_i|=u_i\leq t$ (any $t$) and depth $d_i$.

\end{itemize}

\begin{lemma}\label{Ltree}
With the setup above, $\sum u_id_i\geq s\log_b(s/t)$.
\end{lemma}
\nin
(Of course this is exact if $\T $ is the complete $b$-ary tree of depth $d$
and all leaves have size $2^{-b}s$).

\begin{proof}
Recall that the \emph{relative entropy} between probability
distributions $p$ and $q$ on $[r]$ is
\[
D(p\|q)=\sum p_i\log (q_i/p_i)\leq 0
\]
(the inequality given by the concavity of the logarithm).
We apply this with $p_i=u_i/s$ and
$q_i$
the probability that the ordinary random walk down the tree ends at $u_i$.
In particular $q_i \geq b^{-d_i}$, which, with
nonpositivity of $D(p\|q)$ and the assumption $u_i\leq t$, gives
\begin{eqnarray*}
\sum (u_i/s)d_i\log b &\geq &
\sum(u_i/s)\log (1/q_i)\\
&\geq &\sum(u_i/s)\log (s/u_i) \geq \log (s/t).
\end{eqnarray*}
The lemma follows.\qedhere
\end{proof}

This gives Lemma~\ref{sum |V_i|}
since $\sum |V_i|=\sum_U|U|d(U)$, with $U$ ranging over leaves of $\T$
(and $d(\cdot)$ again denoting depth).\qedhere

\end{proof}

\begin{lemma}\label{m}
The number m of internal nodes of $\T$ is less than $n$.
\end{lemma}

\begin{proof}
A straightforward induction shows that the number of leaves of a rooted tree
is $1+ \sum (b(w)-1)$, where $w$ ranges over internal nodes and $b$ denotes number of children.
The lemma follows since here the number of leaves is at most $n$
(actually at most $3\sqrt{n}$) and each $d(w)$ is at least 2.
\end{proof}

\mn

Recalling that
$
A_i= \{\sigma \in \perm{n}\, : \fit(\sigma,S_i)\ge \eps|S_i|\}
$
and that $\E$ refers to $\pr$, we have 
$
\mean[\fit(\sigma,S_i)] \ge 2\eps|S_i|,
$
which with
\[
\mean[\fit(\sigma,S_i)] \leq \pr(A_i) |S_i| + (1-\pr(A_i))\eps |S_i|
\leq (\pr(A_i) +\eps) |S_i|
\]
gives $\pr(A_i)\ge \eps$ (essentially Markov's Inequality applied to $|S_i|-\fit(\sigma,S_i)$).

\mn

Set $\xi_i=|V_i|\boldsymbol 1_{A_i}$ and $\xi=\sum_i \xi_i$,
and let $Q$ be the event $\{\xi \ge \eps \gL /2\}$. Then
$
\mean[\xi_i] = |V_i|\pr(A_i) \ge  \eps|V_i|,
$
implying $\mean[\xi]=\sum\mean[\xi_i] \ge \eps \gL,$ and (since $\xi_i\leq |V_i|$) $\xi\le \gL$;
so using Markov's Inequality as above gives $\pr(Q) \geq \eps /2$.

\mn

Thus, with $\gs$ chosen from $\perm{n}$ according to $\pr$, we have
\begin{align}\label{Hsigma}
H(\gs) &\leq 1 + (1-\pr(Q))\log n! + \pr(Q) \log |Q| \nonumber\\
&= 1 + \log n! + \pr(Q) \log \mu(Q) \leq 1 + \log n! +(\eps /2) \log \mu(Q)
\end{align}
(recall $\mu$ is the uniform measure on $\perm{n}$).

\mn

Let
\[
\J=\{I\sub [m]: \textstyle \sum_{i\in I} |V_i| \ge \eps \gL /2\}
\]
and, for $I\in \cal J$, let $A_I = \cap_{i \in I} A_i$.  Set
\beq{bc}
\mbox{$b= \eps^2\gd\gb^3/33$}
\enq
(see \eqref{PrBi} for
the reason for the choice of $b$).
We will show, for each $I\in \J$,
\beq{main}
\mu(A_I) \le e^{-b \eps \gL /2},
\enq
which implies
\[
\log \mu(Q) = \log \mu( \cup_{I \in \J} A_I) 
\leq \log |\J| - (b \varepsilon \Lambda \log e) /2 \leq n - (b \varepsilon \Lambda \log e )/2,
\]
the second inequality following from $|\J| \leq 2^{m}$ together with Lemma \ref{m}.  With $c = \eps^\pat{3}\gd\gb^3/150 < (b \eps \log_3 e )/4$, this bounds (for large $n$) the r.h.s.\ of \eqref{Hsigma} 
by
\[
(1- \eps c/2) \log n!,
\]
which proves Theorem~\ref{LMConj} with $\vt = \eps^{4} \delta \beta^3 / 300
= \exp[-\eps^{-O(1)}]$.
\qed

\bn

The rest of our discussion is devoted to the proof of \eqref{main}.  
For a digraph $D\sub L\times R$ with $L,R$ disjoint subsets of $V$, say a pair
$(X,Y)$ of disjoint subsets of $[n]$ with $|X|=|L|$, $|Y|=|R|$ is {\em safe} for $D$
if
\beq{tau}
\fit(\tau, D) < \eps |L||R|/4
\enq
for every bijection $\tau: L\cup R\ra X\cup Y$ with $\tau(L)=X$
(where $\fit(\tau, D)$ has the obvious meaning).
We also say $\sigma\in \perm{n}$ is {\em safe} for $D$ if
$(\gs(L),\gs(R))$ is.
Note that since $S_i$ has density at least 1/2 in $L_i\times R_i$,
the $\gs$'s in $A_i$ are unsafe for $S_i$.

\begin{lemma}\label{Lfit}
Assume the above setup with $|L|+|R|=l$ and $|L|=\gc l$,
and set $\gl=2\gd$ and $\gz = \eps\gd\gc(1-\gc)/4$.
Let $I_1\cup\cdots \cup I_r$ be the natural partition of $X\cup Y$
into intervals of size $\gl l $.
If D is $\gd$-regular
and
\beq{IXY}
|X\cap I_j|= (\gc\gl  \pm \gz)l  ~~\forall j\in [r ],
\enq
then $(X,Y)$ is safe for $D$.
\end{lemma}
\nin
(Of course an {\em interval} of $Z=\{i_1<\cdots < i_u\}$ is one of the sets
$\{i_s\dots i_{s+t}\}$.)

\begin{proof}
For $\tau$ as in the line after \eqref{tau}, let $L_j=L\cap \tau^{-1}(I_j)$ and $R_j=R\cap \tau^{-1}(I_j)$
($j\in [r ]$).
Then
\beq{fitTauD}
|\fit(\tau,D)|\leq
\sum_{1\leq i<j\leq r } ||D\cap (L_i\times R_j)|-|D\cap (L_j\times R_i)|| + \gc(1-\gc)\gl l^2.
\enq
Here the last term is an upper bound on the contribution of
pairs contained in the $I_j$'s:  if $|L_j|=\gc_j|I_j|=\gc_j\gl l$
(so $|R_j| =(1-\gc_j)\gl l$ and $\sum\gc_j = \gc /\gl$),
then
\[
\mbox{$\sum\gc_j(1-\gc_j) \leq \sum \gc_j-(\sum\gc_j)^2/r   = (\gc-\gc^2)/\gl$}
\]
gives
\[
\mbox{$\sum|L_j||R_j| = \sum \gc_j(1-\gc_j)\gl^2l^2 \leq \gc(1-\gc)\gl l^2.$}
\]

On the other hand, regularity and \eqref{IXY}
(which implies $|L_i|> \gd |L|$ ($=\gd\gc l$) since $\gc \gl -\gz>\gc \gd$,
and similarly $|R_i|> \gd |R|$)
give, for all $i\neq j$,
\[
|D\cap (L_i\times R_j)|
=
(d\pm \gd)|L_i||R_j| ,
\]
where $d$ is the density of $D$.
Combining this with \eqref{IXY}
bounds each of the summands in \eqref{fitTauD} by
\[
[(d+\gd)(\gc \gl+\gz)((1-\gc)\gl+\gz)-
(d-\gd)(\gc \gl-\gz)((1-\gc)\gl-\gz)]l^2~~~~
\]
\[
~~~~~~~~~~~~~~~~~~~~~~~~~~~~~~=2[\gl\gz d +\gd(\gc(1-\gc)\gl^2 +\gz^2)]l^2
\]
and the r.h.s.\ of \eqref{fitTauD} by
\[
\left\{2\Cc{r }{2}[\gl\gz d +\gd(\gc(1-\gc)\gl^2 +\gz^2)] + \gc(1-\gc)\gl\right\}l^2
< \eps \gc (1-\gc)l^2/4.
\]
(The main term on the l.h.s.\ is the one with $\gl \gz d$, which, since $r ^{-1}=\gl=2\gd$,
is less than half the r.h.s.  The second and third terms are much smaller
(the second since $\gd$ is much smaller than $\eps$).)
\end{proof}

\begin{cor}\label{LChernoff}
For D and parameters as in Lemma~\ref{Lfit}, and $\sigma$ uniform from $\perm{n}$,
\[   
\Pr(\mbox{$\gs$ is unsafe for $D$}) < 2r \exp[- 2\gz^2l/\gl].
\]   
\end{cor}

\begin{proof}
Let $(X,Y)=(\gs(L),\gs(R))$.
Once we've chosen $X\cup Y$ (determining $I_1\dots I_r$),
$2 \exp[- 2\gz^2l/\gl]$ is the usual Hoeffding
bound \cite[Eq.\ (2.3)]{Hoef}
on the
probability that $X$ violates
\eqref{IXY} for a given $j$.
(The bound may be more familiar when elements of $X\cup Y$ are in $X$ independently, but
also applies to the \emph{hypergeometric} r.v.\ $|X\cap I_j|$;
see e.g.\ \cite[Thm. 2.10 and (2.12)]{JLR}.)
\end{proof}

\begin{proof}[Proof of \eqref{main}.]
Let
\[
B_i= \{\sigma \in \perm{n}\, : \, \mbox{$\sigma$ is unsafe for $S_i$}\}
\]
and $B_I=\cap_{i\in I}B_i$.
Then $A_i\sub B_i$ (as noted above) and (therefore) $A_I\sub B_I$.
Moreover---perhaps the central point---the $B_i$'s are independent,
since $B_i$ depends only on the relative positions of $\gs(L_i)$ and $\gs(R_i)$
within $\gs(V_i)$.

On the other hand, Corollary~\ref{LChernoff}, applied with $D=S_i$ (so $L=L_i$, $R=R_i$, $l=|L_i|+|R_i|$
and $\gc = |L_i|/l\in (\gb,1-\gb)$)
gives
\begin{eqnarray}\label{PrBi}
\Pr(B_i) &<& 2r \exp[- 2\gz^2l/\gl]
< 2r \exp[-\eps ^2 \gd \gb^2 l / 64] \nonumber\\
&<& 2r \exp[-\eps ^2 \gd \gb^3 |V_i| / 32] < e^{-b |V_i|} .
\end{eqnarray}
(Recall $b$ was defined in \eqref{bc}; since we assume $|V_i|$ is large ($|V_i| > \ttt=\sqrt{n}$),
the choice
leaves a little room to absorb the $2r $.)
And of course \eqref{PrBi} and the independence of the $B_i$'s give \eqref{main}.
\end{proof}

\section{Back to the transitive case}\label{sec:transitive}

Theorem~\ref{thm:transitive} is an easy consequence of the next observation.

\begin{lemma}\label{MPC}
Let $\YY$ a random
$m$-subset of $[2m]$ satisfying
\beq{LMhyp}
\E|\{(a,b):a<b,  a\in [2m]\sm \YY,b\in\YY\}| > (\tfrac{1}{2}+\eps)m^2.
\enq
Then $H(\YY) < (1-\eps^2/8)2m$.
\end{lemma}

\mn

To get Theorem~\ref{thm:transitive} from this, let
$T=\{ab:a<b\}$ and, for simplicity, $n=2^k$, and decompose
$T= \bigcup (L_i\times R_i)$ as in \eqref{TD}.
For each $i$, say with $|L_i| $ ($=|R_i|$) $= m_i$, let
$\YY_i\sub [2m_i]$ consist of the indices of positions within $\gs(L_i\cup R_i)$
occupied by $\gs(R_i)$; that is, if $\gs(L_i\cup R_i) =\{j_1<\cdots < j_{2m_i}\}$,
then $\YY_i=\{l:j_l\in \gs(R_i)\}$.
Then Lemma~\ref{MPC} (its hypothesis provided by \eqref{hyp}) gives
\[
H(\YY_i)\leq (1-\eps^2/8)2m_i;
\]
so, since $\gs$ is determined by the $\YY_i$'s, we have
\[
\mbox{$H(\gs)\leq \sum H(\YY_i) \leq (1-\eps^2/8)\sum (2m_i) = (1-\eps^2/8)n\log n.$}  \qquad \qquad \quad \, \qed
\]

\mn
\emph{Remark.}
Note that the $\gO(\eps^2)$ of Theorem~\ref{thm:transitive} is the best
one can do without more fully exploiting \eqref{hyp} (that is, beyond \eqref{LMhyp}
for the $(L_i,R_i,Y_i)$'s, which is all we are using).

\begin{proof}[Proof of Lemma~\ref{MPC}.]
For $a\in [2m]$, set $\pr(a\in \YY) = 1/2 +\gd_a$.  Then
\[
\mbox{$H(\YY) \leq \sum_aH(1/2+\gd_a)\leq \sum_a(1-2\gd_a^2)$}
\]
(where the 2 could actually be $2\log e$);
so it is enough to show
\[
\mbox{$\sum\gd_a^2\geq \eps^2m/8.$}
\]

For a given $m$-subset $Y$ of $[2m]$, we have
\begin{eqnarray*}
f(Y)&:= &\mbox{$|\{(a,b):a<b,  a\in [2m]\sm Y,b\in Y\}|$}\\
&=& \mbox{$\sum_{b\in Y}(b-1) -\C{m}{2} =\sum_{b\in Y} b -\C{m+1}{2}$}.
\end{eqnarray*}
(the first sum counts pairs $(a,b)$ with $a<b$ and $b\in Y$,
and $\C{m}{2}$ is the number of such pairs with $a$ also in $Y$);
so we have
\[
\mbox{$(\frac{1}{2} +\eps)m^2 < \E f(\YY) = \sum (\frac{1}{2}+\gd_b)b -\C{m+1}{2}
=\sum \gd_bb +m^2/2$},
\]
implying $\sum \gd_b b> \eps m^2$.
Combining this with $2m\sum_{\gd_b>0}\gd_b\geq \sum\gd_bb$, we have
$\sum_{\gd_b>0}\gd_b > \eps m/2$ and then, using Cauchy-Schwarz,
\[
\mbox{$\sum\gd_b^2\geq \sum_{\gd_b>0} \gd_b^2\geq \frac{1}{2m}(\eps m/2)^2 =\eps^2m/8.$}\qedhere
\]
\end{proof}

\end{document}